\documentclass[a4paper, 12pt]{amsproc}
\allowdisplaybreaks[1]

\usepackage{amsmath, amsthm, amssymb, mathtools, mathrsfs, stmaryrd}
\usepackage{ascmac}
\usepackage{comment}
\usepackage{bm}
\usepackage{ytableau}

\usepackage{hyperref}

\usepackage{tikz}

\usepackage{graphicx}
\usepackage{here}
\usepackage{time}
\usepackage[abbrev]{amsrefs}

\usepackage{xcolor}
\usepackage[capitalize,nameinlink,noabbrev,nosort]{cleveref}
\hypersetup{
 colorlinks=true,       
 linkcolor=blue,          
 citecolor=blue,        
 filecolor=blue,      
 urlcolor=blue,           
}

\usepackage{fullpage}

\makeatletter
\@namedef{subjclassname@2020}{%
  \textup{2020} Mathematics Subject Classification}
\makeatother

\newtheorem{theoremcounter}{Theorem Counter}[section]

\theoremstyle{definition}

\newtheorem{remark}[theoremcounter]{Remark}

\theoremstyle{plain}
\newtheorem{lemma}[theoremcounter]{Lemma}
\newtheorem{proposition}[theoremcounter]{Proposition}

\newtheorem{conjecture}[theoremcounter]{Conjecture}
\newtheorem{theorem}[theoremcounter]{Theorem}

\numberwithin{equation}{section}

\newcommand{\Q}{\mathbb{Q}}
\newcommand{\R}{\mathbb{R}}
\newcommand{\C}{\mathbb{C}}

\DeclareMathOperator{\ImNew}{Im}
\renewcommand{\Im}{\ImNew}
\DeclareMathOperator{\ReNew}{Re}
\renewcommand{\Re}{\ReNew}

\begin{document}
\author{Takashi Miyagawa}
\address[Takashi Miyagawa]{Onomichi City University,  1600-2 Hisayamada-cho, Onomichi, Hiroshima, 722-8506, Japan} 
\email{miyagawa@onomichi-u.ac.jp}



\begin{abstract}
Upper bounds for the Riemann zeta function $\zeta(s)$ as $\Im(s)\to\infty$ have been extensively studied in connection with the Lindel\"of Hypothesis. In particular, the Lindel\"of Hypothesis is one of the most important open problems in analytic number theory. It is also known to be equivalent to certain mean value estimates, which provide a fundamental connection between pointwise upper bounds and integral mean values of zeta-functions.

In this paper, we consider an analogue of the Lindel\"of Hypothesis for the Barnes multiple zeta function 
$\zeta_r (s,a,(w_1,\dots,w_r)) = \sum_{m_1=0}^\infty \cdots \sum_{m_r=0}^\infty (a+m_1 w_1+\cdots+m_r w_r)^{-s} $ for $\Re(s)>r$, with $a,w_1,\dots,w_r>0$,
and establish equivalent conditions in terms of integral mean values. We treat the rank-one case and, under a polynomial growth assumption on the derivative, the general-rank case. The $\Q$-rank of $\langle w_1,\dots,w_r\rangle$ specifies the region considered in our statements; the equivalence argument itself does not establish that its left boundary is a critical line.
\end{abstract}


\title{Analogues of the Lindel\"of Hypothesis for the Barnes multiple zeta function and related problems}

\maketitle


\section{Introduction}
Let $r $ be a positive integer, $s=\sigma+it \ (\sigma, t \in \R) $ a complex variable, $a>0$, and $w_1,\ldots,w_r>0$.
The Barnes multiple zeta function, denoted by $\zeta_r(s,a,(w_1, \dots,w_r))$ and introduced in \cite{Barnes1899, Barnes1901, Barnes1904}, is defined as  
\begin{align}\label{zeta_r}
 \zeta_r (s,a,(w_1,\dots,w_r))
 =\sum_{m_1=0}^\infty \cdots \sum_{m_r=0}^\infty
  \frac{1}{(a+m_1 w_1+\cdots+m_r w_r)^s}.
\end{align}
The series on the right-hand side converges absolutely for $\mathrm{Re}(s) > r$ and can be meromorphically continued to the entire complex $s$-plane. 
The function has at most simple poles, all of which lie among $s=1,\dots,r$.

The Barnes multiple zeta function was originally introduced by Barnes in the course of developing the theory of multiple gamma functions, which can be regarded as higher-dimensional analogues of the classical gamma function.  
This construction is based on Lerch's formula for the Hurwitz zeta function $\zeta_H(s,a)$, which expresses the logarithm of the classical gamma function as
\[
\log \Gamma(a) = \zeta_H'(0,a) + \frac{1}{2}\log 2\pi.
\]
In particular, the multiple gamma function can be defined in terms of the derivative of the Barnes multiple zeta function at $s=0$ as
\[
\log \Gamma_r(a,\mathbf{w})
= \left. \frac{\partial}{\partial s} \zeta_r(s,a,\mathbf{w}) \right|_{s=0}
+ \rho_r(\mathbf{w}),
\]
where $\rho_r(\mathbf{w})$ is a normalization constant depending only on $\mathbf{w}$.  
Since then, Barnes-type zeta functions have been studied extensively in connection with special functions and analytic number theory.


Next, we recall some fundamental growth estimates for the Riemann zeta-function
	\[
	\zeta(s) = \sum_{n=1}^\infty \frac{1}{n^s}\qquad (\Re(s)>1).
	\]
	These estimates concern the growth of $ \zeta(\sigma+it) $ 
	as the imaginary part $t$ tends to infinity.
	In particular, the growth of $\zeta(1/2+it)$ as $t\to\infty$ has attracted much attention.
	The classical convexity bound
	\[
	\zeta(\sigma + it) \ll |t|^{(1-\sigma)/2 +\varepsilon} 
	\quad (0 \leq \sigma \leq 1, |t| \geq 2)
	\]
	is known; in particular, for $\sigma=1/2$ it gives $\zeta(1/2+it)\ll t^{1/4+\varepsilon}$.
	This result is obtained using the functional equation of $ \zeta(s) $ and the Phragm\'en-Lindel\"of convexity principle. 
	A classical finite-sum approximation for $\zeta(s)$ is
	\begin{equation}\label{order_of_zeta}
		\zeta(s) = \sum_{n \leq x} \frac{1}{n^s} - \frac{x^{1-s}}{1-s} 
		+ O(x^{-\sigma}) \qquad (x \rightarrow \infty),
	\end{equation}
	uniformly for $\sigma_0\leq\sigma\leq 2$, $0<\sigma_0<2$, and $2\leq |t|<2\pi x/C$, where $C>1$ is fixed. This formula is useful in the study of the critical strip of $\zeta(s)$.
	Hardy and Littlewood improved this to
\begin{equation}\label{Hardy-Littlewood}
	\zeta \left( \frac12 + it \right) \ll t^{1/6 + \varepsilon}
\end{equation}
using exponential-sum estimates applied to \eqref{order_of_zeta}.
Furthermore, the problem of improving upper bounds for
$\zeta(1/2+it)$ with respect to $t$ has been extensively studied
by many mathematicians. To mention a few results, in 1986,
Bombieri and Iwaniec established the bound
$\zeta(1/2 + it) \ll t^{9/56 + \varepsilon}$.
Since then, several refinements have been obtained, and Huxley proved $\zeta(1/2 + it) \ll t^{32/205 + \varepsilon}$ in 2005.
Furthermore, in 2017,
\[
\zeta\left( \frac12 + it \right) \ll t^{13/84 + \varepsilon}
\]
was proved by Bourgain (see \cite{Bourgain2017}).
This is a major improvement on the classical bounds.

Improvements in order estimates for
$\zeta(1/2+it)$ are still ongoing, and the conjectured upper bound is
\begin{equation}\label{Lindelof_hypothesis}
	\zeta \left( \frac12 + it \right) \ll t^{\varepsilon}
\end{equation}
for any $\varepsilon >0$.
This conjecture is called the Lindel\"of hypothesis.
Furthermore, it is well known that the Lindel\"of hypothesis
(\ref{Lindelof_hypothesis}) is equivalent to
\begin{equation}\label{Lindelof_hypothesis_int}
	\int_2^T
	\left|
	\zeta\left(\frac12+it\right)
	\right|^{2k}
	dt
	= O(T^{1+\varepsilon})
\end{equation}
for any $k\in\mathbb N$ and any $\varepsilon>0$
(see, for example,
\cite[Chapter 8]{Ivic1985}
and
\cite[Chapter 8]{Titchmarsh1986}).

\bigskip

For convenience in this paper, we introduce the following boldface symbols to represent index tuples:
\begin{align*}
 &\mathbf{m}=(m_1,\dots,m_r), \quad
  \mathbf{n}=(n_{1},\dots,n_{r}), \quad\\
 &\mathbf{w}=(w_1,\dots,w_r), \ \, \quad
  \mathbf{1}=(1,\dots,1).
\end{align*}
We assume that $w_1,\dots,w_r\in \mathbb{R}_{>0}$ throughout this paper.
Using this notation, \eqref{zeta_r} 
can be rewritten as  
\begin{align}
 \zeta_r(s,a,\mathbf{w})
 =\sum_{m_1,\dots,m_r\ge 0}
 \frac{1}{(a+\mathbf{m}\cdot\mathbf{w})^s}.
\end{align}
As an analogue of the classical formula (\ref{order_of_zeta}), the following result was obtained in \cite{MiyagawaMurahara2025}.
    
	\medskip

\begin{proposition}[Theorem 1.2 (v2) in \cite{MiyagawaMurahara2025}]\label{finite_sum}
Let $r-1<\sigma_1<\sigma_2$, $x\ge 1$, and $C>1$. 
If $s=\sigma+it \in \C$ with $\sigma_1<\sigma<\sigma_2$ and $2\leq |t|\le 2\pi x/C$, then
\begin{align*} 
 \zeta_r(s,a,\mathbf{w})
 &=\sum_{0\le m_1\le x}
  \cdots 
  \sum_{0\le m_r\le x}
  \frac{1}{(a+\mathbf{m}\cdot\mathbf{w})^{s}}\\
 &\quad -
 \sum_{\substack{E\subseteq\{1,\dots,r\}\\E\ne\emptyset}}
 (-1)^{\#E}
 \frac{(a+x\sum_{j\in E}w_j)^{r-s}}{(s-1)\cdots (s-r) w_1\cdots w_{r}}
 +O(x^{r-1-\sigma})
\end{align*}
as $x\rightarrow \infty$.
\end{proposition}
    \medskip
	

Furthermore, by applying the finite sum approximation in
Proposition \ref{finite_sum} with $x=t$ and $C=2$, we immediately obtain
the following upper bound for
$\zeta_r(\sigma+it,a,\mathbf{w})$ as $t \to \infty$:

\begin{proposition}[Theorem 1.5 in \cite{miyagawa2026}]\label{thm:Barnes_bounds}
Let $s=\sigma+it$ with $t\ge 2$, $a>0$, and
$\mathbf{w}\in\mathbb{R}_{>0}^r$.
Then the following bounds hold:
\[
\zeta_r(\sigma+it,a,\mathbf{w}) \ll 
\begin{cases}
1 
& \text{if } \quad   \sigma> r,\\
\log t
& \text{if } \quad  \sigma=r,\\
t^{\,r-\sigma}
& \text{if } \quad  r-1< \sigma < r,
\end{cases}
\]
as $t \to \infty$, uniformly for $\sigma$ in any fixed compact subinterval
of each region.
The implied constants depend on $a$, $r$, $\mathbf{w}$, and the chosen strip.
\end{proposition}

\section{Statements of Main Results}

As an analogue of the equivalence between
\eqref{Lindelof_hypothesis} and \eqref{Lindelof_hypothesis_int}
for the classical Lindel\"of Hypothesis, we obtain the following theorems for the Barnes multiple zeta function.
	

    \medskip
    

\begin{theorem}\label{Lindelof_d=1}
Suppose that $a>0$ and that
$ \mathbf{w}=\lambda (p_1,\ldots,p_r) $,
where $\lambda>0$, $p_j\in \mathbb{N}$, and
$\gcd(p_1,\ldots,p_r)=1$.
For each fixed $\sigma$ with $r-1/2\leq\sigma\leq r$,
\[
\zeta_r(\sigma+it,a,\mathbf{w})
=O(t^\varepsilon)
\qquad
\left(r-\frac{1}{2}\leq \sigma\leq r\right)
\]
holds for every $\varepsilon>0$ if and only if
\[
\int_2^T
\left|
\zeta_r(\sigma+it,a,\mathbf{w})
\right|^{2k}dt
=
O(T^{1+\varepsilon})
\qquad
\left(r-\frac{1}{2}\leq \sigma\leq r\right)
\]
holds for any $k\in \mathbb{N}$ and any $\varepsilon>0$.
Here the implied constants may depend on
$\varepsilon$, $k$, $\sigma$, $a$, and $\mathbf{w}$.
\end{theorem}

\begin{theorem}\label{Lindelof_zeta_r_rank_d}
Suppose that $a>0$, $w_1,\dots,w_r>0$, and
$ \dim_{\mathbb Q}\langle w_1,\dots,w_r\rangle=d \ (1\le d \le r)$.
Assume that the derivative $\zeta_r'(s,a,\mathbf w)$ with respect to $s$ has polynomial growth as $t\to\infty$ in the fixed strip $r-d/2\leq\sigma\leq r$.
Then, for each fixed $\sigma$ in this strip,
\[
\zeta_r(\sigma+it,a,\mathbf w)
=
O(t^\varepsilon)
\qquad
\left(r-\frac d2\leq \sigma\leq r\right)
\]
holds for every $\varepsilon>0$ if and only if
\[
\int_2^T
\left|
\zeta_r(\sigma+it,a,\mathbf w)
\right|^{2k}dt
=
O(T^{1+\varepsilon})
\qquad
\left(r-\frac d2\leq \sigma\leq r\right)
\]
holds for any $k\in\mathbb N$ and any $\varepsilon>0$.
Here the implied constants may depend on
$\varepsilon$, $k$, $\sigma$, $a$, and $\mathbf w$.
\end{theorem}

\medskip

\begin{remark}
The polynomial growth assumption in Theorem \ref{Lindelof_zeta_r_rank_d} is used in the implication from moment estimates to pointwise estimates. Small denominators in a functional-equation representation do not by themselves prove that this assumption is necessary.
For example, if $w_1=1$ and $w_2=\theta>0$, where $\theta$ is a Liouville number, factors of the form $(e^{2\pi i n\theta}-1)^{-1}$ can be very large. However, to infer a growth estimate for the zeta function from such a representation, its convergence and any cancellation between its terms must first be justified (see \cite{Miyagawa2026-0}).
The proof of the equivalence below uses no further rank-specific property. In particular, it does not establish the optimality of the boundary $\sigma=r-d/2$.
\end{remark}

\section{Auxiliary Lemmas and Proofs of the Main Theorems}
In this section, we collect several auxiliary lemmas that will be used in the proofs of the main theorems.

\medskip

\medskip

\begin{lemma}\label{lem:local_persistence}
Suppose that
$\zeta_r'(\sigma+it,a,\mathbf w)\ll t^B$
holds for $t\geq2$ in a fixed vertical strip, with $B\geq0$.
If
$|\zeta_r(\sigma+it_0,a,\mathbf w)|\geq t_0^\eta$
for some $\eta>0$ and sufficiently large $t_0$, then
\[
|\zeta_r(\sigma+it,a,\mathbf w)|
\geq \frac12 t_0^\eta
\]
holds whenever
$|t-t_0|\leq c\min\{1,t_0^{\eta-B}\}$
with a sufficiently small constant $0<c\leq1$.
\end{lemma}

\begin{proof}
Put
\[
F(t)=\zeta_r(\sigma+it,a,\mathbf w).
\]
Then
\[
F'(t)=i\zeta_r'(\sigma+it,a,\mathbf w),
\]
and hence, by the assumption,
\[
|F'(t)|\ll t^B
\]
in the strip under consideration. For $|t-t_0|\leq1$ and sufficiently large $t_0$, we have $t\asymp t_0$. The fundamental theorem of calculus therefore gives
\[
|F(t)-F(t_0)|
\leq
\sup_{|\xi-t_0|\leq1} |F'(\xi)|\, |t-t_0|
\ll
t_0^B |t-t_0|.
\]
If
\[
|t-t_0|\leq c\min\{1,t_0^{\eta-B}\}
\]
with $c>0$ sufficiently small, then
\[
|F(t)-F(t_0)|\leq \frac12 t_0^\eta.
\]
Thus
\[
|F(t)|
\geq
|F(t_0)|-|F(t)-F(t_0)|
\geq
t_0^\eta-\frac12 t_0^\eta
=
\frac12 t_0^\eta.
\]
This proves the lemma.
\end{proof}

\medskip

\begin{remark}
$\zeta_r(s,a,\mathbf{1})$ is referred to as the Hurwitz multiple zeta function, and it can be expressed as a linear combination of single Hurwitz zeta functions (see \cite{SrivastavaChoi2001} p. 86):  
\begin{align}  
  \zeta_r(s,a,\mathbf{1})=\sum_{j=0}^{r-1}p_{r,j}(a)\zeta_H(s-j, a). 
  \label{HMZ-HZ}
\end{align}
In this formula, $p_{r,j}(a)$ is given by  
\begin{align*}  
 p_{r,j}(a)=\frac{1}{(r-1)!}\sum_{l=j}^{r-1}(-1)^{r+1-j}\binom{l}{j}S(r,l+1) a^{l-j},  
\end{align*}  
where $S(r,l+1)$ denotes the signed Stirling number of the first kind, with the convention
\[
\prod_{j=0}^{r-1}(s-j)=\sum_{l=0}^{r}S(r,l)s^l.
\]
\end{remark}

\medskip

\begin{proof}[Proof of Theorem \ref{Lindelof_d=1}]
Since $w_j=\lambda p_j$, we have for $\Re(s)>r$
\[
\zeta_r(s,a,\mathbf{w})
=
\lambda^{-s}
\sum_{m_1,\ldots,m_r\geq 0}
\frac{1}{(a/\lambda+p_1m_1+\cdots+p_rm_r)^s}.
\]
Put $b=a/\lambda$. For each integer $n\geq 0$, let
\[
A(n)=\#\{(m_1,\ldots,m_r)\in \mathbb{Z}_{\geq 0}^r
\mid
p_1m_1+\cdots+p_rm_r=n\}.
\]
Then
\[
\zeta_r(s,a,\mathbf{w})
=
\lambda^{-s}\sum_{n=0}^{\infty}\frac{A(n)}{(n+b)^s}.
\]
It is well known that $A(n)$ is a quasi-polynomial of degree $r-1$.
Hence the above Dirichlet series can be expressed as a finite linear
combination of Hurwitz zeta-functions:
\[
\zeta_r(s,a,\mathbf{w})
=
(\lambda q)^{-s}
\sum_{\nu=0}^{q-1}
\sum_{\ell=0}^{r-1}
c_{\nu,\ell}\,
\zeta_H(s-\ell,\beta_\nu),
\]
where $q$ may be taken to be a common multiple of $p_1,\ldots,p_r$,
$\beta_\nu=(\nu+b)/q>0$, and the constants $c_{\nu,\ell}$ depend only on $p_1,\ldots,p_r$, $q$, and $b$.
Indeed, on writing $n=q m+\nu$, the polynomial $A(qm+\nu)$ can be expanded in powers of $m+\beta_\nu$; the denominator contributes the factor $q^{-s}$. The identity then extends meromorphically to all $s$.

The implication from the Lindel\"of-type bound to the moment estimates is
immediate. Indeed, replacing $\varepsilon$ by
$\varepsilon/(2k)$ in the pointwise estimate gives
\[
\left|\zeta_r(\sigma+it,a,\mathbf{w})\right|^{2k}
\ll t^\varepsilon,
\]
and therefore
\[
\int_2^T
\left|\zeta_r(\sigma+it,a,\mathbf{w})\right|^{2k}dt
\ll T^{1+\varepsilon}.
\]
Conversely, assume that the moment estimate holds for every
$k\in\mathbb N$ and every $\varepsilon>0$.
Using the above finite linear representation in terms of Hurwitz
zeta-functions, together with the standard polynomial growth estimates
for the Hurwitz zeta-function in vertical strips, we obtain a polynomial
growth estimate for
$
\zeta_r(s,a,\mathbf{w})
$
and, by Cauchy's integral formula on small circles in a slightly larger strip, also for its derivative with respect to $t$ in the strip
$r-1/2\leq \sigma\leq r$.
By Lemma \ref{lem:local_persistence}, a value at least $t_0^\eta$ persists on an interval of length comparable to $\min\{1,t_0^{\eta-B}\}$, for a fixed derivative-growth exponent $B$. Its contribution to the $2k$-th moment is at least a constant times $t_0^{2k\eta+\min\{0,\eta-B\}}$. Choosing $k$ so that this exponent exceeds $2$ contradicts the assumed moment bound with $\varepsilon=1$ and $T=2t_0$, as detailed in the proof of Theorem \ref{Lindelof_zeta_r_rank_d} below. Hence
\[
\zeta_r(\sigma+it,a,\mathbf{w})=O(t^\varepsilon)
\]
for every $\varepsilon>0$.
This proves the equivalence.
\end{proof}


	\medskip
	
	\begin{proof}[Proof of Theorem \ref{Lindelof_zeta_r_rank_d}]
Fix $\sigma$ with $r-d/2\leq\sigma\leq r$. Assume that
\[
\zeta_r(\sigma+it,a,\mathbf{w})
=
O\!\left(|t|^{\varepsilon/(2k)}\right)
\qquad
 \left( r-\frac{d}{2}\le \sigma\le r \right)
\]
for any $\varepsilon>0$ and $k\in\mathbb{N}$.  Then
\[
\int_{2}^T
|\zeta_r(\sigma+it,a,\mathbf{w})|^{2k}dt
\ll
\int_{2}^T t^\varepsilon dt
\ll
T^{1+\varepsilon}.
\]
Conversely, assume that the moment estimate holds for every
$k\in\mathbb N$ and every $\varepsilon>0$.
Suppose, to the contrary, that the Lindel\"of-type estimate fails.
Then there exist $\eta>0$ and a sequence $t_n\to\infty$ such that
\[
\left|\zeta_r(\sigma+it_n,a,\mathbf w)\right|
\geq t_n^\eta .
\]
By the assumed polynomial growth of $\zeta_r'(s,a,\mathbf w)$ in
vertical strips, there exists a constant $B>0$ such that
\[
\zeta_r'(\sigma+it,a,\mathbf w)\ll t^B
\]
in the strip under consideration. Hence, for
$
|t-t_n|\leq c\min\{1,t_n^{\eta-B}\}
$
with a sufficiently small constant $0<c\leq1$, Lemma \ref{lem:local_persistence} gives
\[
\left|\zeta_r(\sigma+it,a,\mathbf w)\right|
\geq \frac12 t_n^\eta .
\]
Therefore,
\[
\int_{t_n-c\min\{1,t_n^{\eta-B}\}}^{t_n+c\min\{1,t_n^{\eta-B}\}}
\left|\zeta_r(\sigma+it,a,\mathbf w)\right|^{2k}dt
\gg
t_n^{2k\eta+\min\{0,\eta-B\}}.
\]
The interval above is contained in $[2,2t_n]$ for all sufficiently large $n$. Taking $\varepsilon=1$ in the moment assumption and choosing $k$ so that $2k\eta+\min\{0,\eta-B\}>2$, we contradict the assumed bound
\[
\int_2^T
\left|\zeta_r(\sigma+it,a,\mathbf w)\right|^{2k}dt
\ll T^{1+\varepsilon}.
\]
Thus the Lindel\"of-type estimate follows.
\end{proof}



\medskip

\medskip

\begin{remark}\label{rem:possible_half_bound}
By Proposition \ref{finite_sum}, we have
\[
\zeta_r(s,a,\mathbf{w})
=
\sum_{0\le m_1,\dots,m_r\le x}
(a+\mathbf{m}\cdot\mathbf{w})^{-s}
+
O(x^{r-1-\sigma})
\]
for $x\asymp t$ satisfying the height restriction in Proposition \ref{finite_sum} and $r-1<\sigma<r$. Here the correction term in that proposition is $O(t^{-\sigma})$, since its denominator is of order $t^r$, and is absorbed by $O(x^{r-1-\sigma})$.
The trivial estimate for the finite sum yields
\[
\zeta_r(\sigma+it,a,\mathbf{w})
\ll
t^{r-\sigma}.
\]
However, the oscillatory factor
\[
(a+\mathbf{m}\cdot\mathbf{w})^{-it}
=
e^{-it\log(a+\mathbf{m}\cdot\mathbf{w})}
\]
suggests that additional cancellation may occur.
Indeed, after fixing
$m_2,\dots,m_r$,
the inner sum with respect to $m_1$ can be regarded as a
one-dimensional exponential sum of the form
\[
\sum_{0\le m_1\le x}
e^{-it\log(a+B+m_1w_1)}
(a+B+m_1w_1)^{-\sigma},
\]
where
\[
B=m_2w_2+\cdots+m_rw_r.
\]
This suggests the possible bound
\[
\zeta_r(\sigma+it,a,\mathbf{w})
\ll
t^{1/2}
\qquad
\left(r-1<\sigma\le r-\frac12\right),
\]
which would improve the trivial estimate
\[
\zeta_r(\sigma+it,a,\mathbf{w})
\ll
t^{r-\sigma}
\]
for $r-1<\sigma<r-1/2$; at $\sigma=r-1/2$ the two exponents coincide.
A rigorous proof would require a more delicate analysis of the
associated exponential sums, such as dyadic decomposition and
van der Corput type estimates.
If proved, this pointwise estimate would also give the mean-square bound
\[
 \int_1^T |\zeta_r(\sigma+it,a,\mathbf w)|^2dt \ll T^2
 \qquad \left(r-1<\sigma<r-\frac{1}{2}\right),
\]
which would improve the bound obtained
directly from $\zeta_r(\sigma+it,a,\mathbf w)\ll t^{r-\sigma}$.
These considerations suggest that stronger pointwise and mean-square
bounds may hold in the region
$r-1<\sigma\le r-1/2$.
However, a rigorous proof appears to require a substantially more
delicate analysis of the associated exponential sums, and we leave this
problem for future work.
\end{remark}


\medskip

\bibliographystyle{plain}
\bibliography{References} 

\end{document}